\documentclass[12pt,a4paper]{amsart}

\usepackage[utf8]{inputenc}
\usepackage[T1]{fontenc}
\usepackage{microtype}
\usepackage{lmodern}
\usepackage{standard}
\usepackage{mathtools}
\usepackage{amsrefs}
\usepackage{enumerate}
\usepackage[usenames,dvipsnames]{xcolor}
\usepackage[colorlinks=true,linkcolor=Red,citecolor=Green]{hyperref}

\usepackage{tikz}
\usetikzlibrary{decorations.markings}

\author[S. Coriasco]{Sandro Coriasco}
\address{Dipartimento di Matematica ``G. Peano''\newline\indent Universit\`a degli Studi di Torino\newline\indent V. C. Alberto, n. 10, I-10126 Torino, Italy}
\email{sandro.coriasco[AT]unito.it}

\author[M. Doll]{Moritz Doll}
\address{Department 3 -- Mathematics\newline\indent University of Bremen\newline\indent Bibliotheksstr. 5, D-28359 Bremen, Germany}
\email{doll[AT]uni-bremen.de}

\title{Weyl Law on Asymptotically Euclidean Manifolds}
\date{\today}

\begin{document}

\begin{abstract}
    We study the asymptotic behaviour of the eigenvalue counting function
    for self-adjoint elliptic linear operators defined through classical weighted symbols of order $(1,1)$,
    on an asymptotically Euclidean manifold.
    We first prove a two term Weyl formula, improving previously known remainder estimates.
    Subsequently, we show that under a geometric assumption on the Hamiltonian flow at infinity there is a 
    refined Weyl asymptotics with three terms.
    The proof of the theorem uses a careful analysis of the flow behaviour in the corner component of the 
    boundary of the double compactification of the cotangent bundle.
    Finally, we illustrate the results by analysing the operator $Q=(1+|x|^2)(1-\Delta)$ on $\RR^d$.
\end{abstract}

\maketitle

\section{Introduction}
Let $(X,g)$ be a $d$-dimensional asymptotically Euclidean manifold.
On $X$ we consider a self-adjoint positive operator $P$, elliptic in the $\SG$-calculus of order $(m,n)$ with $m, n\in (0,\infty)$.
By the compact embedding of weighted Sobolev spaces, the resolvent is compact and hence the spectrum of $P$ consists of a sequence of eigenvalues
\begin{align*}
    0 < \lambda_1 \leq \lambda_2 \leq \dotsc \to +\infty.
\end{align*}
The goal of this article is to study the Weyl law of $P$, that is, the asymptotics of its counting function,
\begin{equation}\label{eq:cf}
	N(\lambda)= \#\{ j \colon \lambda_j < \lambda\}.
\end{equation}

Hörmander~\cite{Hormander68} proved, for a positive elliptic self-adjoint classical pseudodifferential operator of order $m > 0$
on a compact manifold, the Weyl law
\[
	N(\lambda)=\gamma\cdot \lambda^\frac{d}{m}+O(\lambda^\frac{d-1}{m}), \quad \lambda\to+\infty.
\]
It was pointed out that, in 
general, this is the \emph{sharp remainder estimate}, since the exponent of $\lambda$ in the remainder term
cannot be improved for the Laplacian on the sphere.
It was subsequently shown by Duistermaat and Guillemin~\cite{DuGu75} that under a geometric assumption
there appears an additional term $\gamma' \lambda^{(d-1)/m}$ and 
the remainder term becomes $o(\lambda^{(d-1)/m})$.

In the case of $\SG$-operators on manifolds with ends, the leading order of the Weyl asymptotics was found by 
Maniccia and Panarese~\cite{MaPa02}.
Battisti and Coriasco~\cite{BaCo11} improved the remainder estimate to $O(\lambda^{d/\max\{m,n\} - \ep})$ for some $\ep > 0$.
For $m \not = n$, Coriasco and Maniccia~\cite{CoMa13} proved the general sharp remainder estimate.

In Theorem~\ref{thm:sharp-weyl}, we prove the analogue of Hörmander's result for $m = n$.
This provides a more precise remainder term compared to the earlier result given in \cite{BaCo11}.
If the \emph{geodesic flow at infinity} generated by the \emph{corner component} $p_\psie$ 
of the principal symbol of $P$ is sufficiently generic, we have an even more refined estimate, parallel to the Duistermaat--Guillemin theorem, described in Theorem~\ref{thm:refined}.

\begin{theorem}\label{thm:sharp-weyl}
    Let $P \in \calcsg^{m,m}_\cl(X)$ be a self-adjoint, positive, elliptic $\SG$-classical 
    pseudodifferential operator on an asymptotically Euclidean manifold $X$, and $N(\lambda)$ 
    its associated counting function. Then, the corresponding Weyl asymptotics reads as
    \begin{align*}
        N(\lambda) = \gamma_2 \lambda^{\frac dm}\log\lambda + \gamma_1 \lambda^{\frac dm}+ O(\lambda^{\frac{d-1}m}\log\lambda).
    \end{align*}
    The coefficients $\gamma_j$, $j=1,2$, are given by 
    \begin{align*}
        \gamma_2 &= \frac{\TR(P^{-\frac{d}{m}})}{m\cdot d},
        \\
        \gamma_1 &= \frac{\wTR_{x,\xi}(P^{-\frac dm})}{d} - \frac{\TR(P^{-\frac dm})}{d^2},
    \end{align*}
    where $\TR$ and $\wTR_{x,\xi}$ are suitable trace operators on the algebra
    of $\SG$-operators on $X$.
\end{theorem}
\begin{theorem}\label{thm:refined}
    Let $P \in \calcsg^{m,m}_\cl(X)$ and $N(\lambda)$ be as in Theorem \ref{thm:sharp-weyl} above.
    Denote by $p_\psie$ the corner component of the principal symbol of $P$.
    If the set of periodic orbits of the Hamiltonian flow of $\hamvf_{f}$, $f=(p_\psie)^\frac{1}{m}$,
    has measure zero on $\Wpe$, then we have the estimate
    \begin{align*}
        N(\lambda) = \gamma_2 \lambda^{\frac dm}\log\lambda + \gamma_1 \lambda^\frac{d}{m} + 
        \gamma_0 \lambda^{\frac{d-1}{m}}\log\lambda + o(\lambda^{\frac{d-1}{m}}\log\lambda),
    \end{align*}
    with the coefficients $\gamma_2$ and $\gamma_1$ given in Theorem \ref{thm:sharp-weyl}, and
    \[\gamma_0 = \dfrac{\TR(P^{-\frac{d+1}m})}{m\cdot (d-1)}.\]
\end{theorem}
\begin{remark}
    The trace operators $\TR$ and $\wTR_{x,\xi}$ appearing in Theorems \ref{thm:sharp-weyl} and \ref{thm:refined}
    were introduced in \cite{BaCo11}.
    The coefficient $\gamma_0$ can be calculated as the Laurent coefficient of order $-2$ at $s = d-1$ of $\zeta(s)$,
    the spectral $\zeta$-function associated with $P$.
\end{remark}
\begin{remark}
    To our best knowledge, this is the first result of a logarithmic Weyl law with the remainder being one order lower than the 
    leading term (we refer to \cite{BaCo11} for other settings with logarithmic Weyl laws).
\end{remark}
Next, we apply our results to
the model operator $P$ associated with the symbol $p(x,\xi)=\ang{x}\!\cdot\!\ang{\xi}$, $\ang{z}=\sqrt{1+|z|^2}$, $z\in\RR^d$,
that is, $P=\ang{\cdot}\sqrt{1-\Delta}$. In particular, we observe that the condition on the underlying Hamiltonian flow in
Theorem \ref{thm:refined} is not satisfied, and compute explicitly the coefficients $\gamma_1$ and $\gamma_2$. 
\begin{theorem}\label{thm:explicit-example}
    Let $P = \ang{\cdot} \ang{D} \in \calcsg^{1,1}(\RR^d)$. Then,
    \begin{align*}
        N(\lambda) = \gamma_2 \lambda^d\log\lambda + \gamma_1 \lambda^d + O(\lambda^{d-1}\log\lambda).
    \end{align*}
    Here, the coefficients are
    \begin{align*}
        \gamma_2 &= \frac{[\vol(\Sph^{d-1})]^2}{(2\pi)^{d}} \cdot \frac{1}{d} ,
        \\
        \gamma_1 &= \frac{[\vol(\Sph^{d-1})]^2}{(2\pi)^{d}} \cdot \!\left[\Psi\!\left(\frac{d}{2}\right)+\gamma-\frac{1}{d^2}\right]\!,
    \end{align*}
    where $\displaystyle\gamma=\lim_{n\to+\infty}\left(\sum_{k=1}^n\frac{1}{k}-\log n\right)$ is the 
    Euler-Mascheroni constant and \[\Psi(x)=\dfrac{d}{dx}\log\Gamma(x)\] is the digamma function.
\end{theorem}
This implies that the Weyl asymptotics of the operator \[Q = (1 + |x|^2) (1 - \Delta)\] is given by
\begin{align*}
    N(\lambda) = \frac{\gamma_2}{2} \lambda^{\frac d2}\log\lambda + \gamma_1 \lambda^{\frac d2} + O(\lambda^{\frac{d-1}2}\log\lambda),
\end{align*}
with the same coefficients given in Theorem \ref{thm:explicit-example} above.

The paper is organized as follows. In Section \ref{sec:sg} we fix most of the notation used throughout the
paper and recall the basic elements of
the calculus of $\SG$-classical pseudodifferential operators, the associated wave-front set, and 
the computation of the parametrix of Cauchy problems for $\SG$-hyperbolic operators of order $(1,1)$. 
In Section \ref{sec:wt} we consider the wave-trace of a $\SG$-classical operator $P$ of order
$(1,1)$. Section \ref{sec:zfandwt} is devoted to study the relation between the wave-trace and the spectral
$\zeta$-function of $P$.
In Section \ref{sec:proofs} we prove our main Theorems \ref{thm:sharp-weyl} and \ref{thm:refined}, 
while in Section \ref{sec:ex} we examine the example given by the model operator $P=\ang{\cdot}\ang{D}$,
and prove Theorem \ref{thm:explicit-example}.
We conclude with a short appendix on asymptotically Euclidean manifolds and a few more remarks about
aspects of the proofs of the main results.

\subsection*{Acknowledgements}
We would like to thank R. Schulz for many helpful discussions and various remarks on the manuscript.

\section{SG-Calculus on \texorpdfstring{$\mathbb{R}^d$}{Rd}}\label{sec:sg}
The Fourier transform $\mathcal{F} : \Sw(\RR^d) \to \Sw(\RR^d)$ is defined by
\begin{align*}
    (\mathcal{F} u)(\xi) = \widehat{u}(\xi)=\int e^{-ix\xi}\,u(x)\,dx, u \in \Sw(\RR^d), 
\end{align*}
and extends by duality to a bounded linear operator $\mathcal{F} : \Sw'(\RR^d) \to \Sw'(\RR^d)$.
The set of pseudodifferential operators $A = a^w(x,D) = \Op^w(a) : \Sw(\RR^d) \to \Sw'(\RR^d)$ on $\RR^d$
with Weyl symbol $a \in \Sw'(\RR^{2d})$ can be defined through the Weyl-quantization%
\footnote{The formula 
involving integrals only holds true for $a \in \Sw(\RR^{2d})$, but the quantization can be extended to any
$a \in \Sw'(\RR^{2d})$ using the 
Fourier transform, pull-back by linear transformations, and the Schwartz kernel theorem.}
\begin{align*}
    Au(x) = (2\pi)^{-d}\iint e^{i(x-y)\xi} a( (x+y)/2, \xi) u(y) dy\,d\xi, \quad u \in \Sw(\RR^d).
\end{align*}

A smooth function $a \in \Sm(\RR^d\times \RR^d)$ is a SG-symbol of order 
$(m_\psi,m_e) \in \RR^2$, and we write $a \in \symbsg^{m_\psi,m_e}(\RR^{2d})$, if for all
multiindices $\alpha,\beta \in \NN^d$ there exists $C_{\alpha\beta}>0$ such that, for all $x,\xi\in\RR^d$,
\begin{align}\label{eq:sg-estimate}
    \abs{\pa_x^\alpha \pa_\xi^\beta a(x,\xi)} \leq C_{\alpha\beta} \ang{\xi}^{m_\psi-|\beta|} \ang{x}^{m_e - |\alpha|}.
\end{align}
The space $\symbsg^{m_\psi,m_e}(\RR^{2d})$ becomes a Fr\'echet space with the seminorms being the best constants 
in \eqref{eq:sg-estimate}. The space of all $\SG$-pseudodifferential operators of order $(m_\psi,m_e)$ is denoted by
\begin{align*}
    \calcsg^{m_\psi,m_e}(\RR^d) = \{ \Op^w(a) \colon a \in \symbsg^{m_\psi,m_e}(\RR^{2d})\}.
\end{align*}
We have the following properties (we refer to, e.g., \cite{Cordes} and \cite{NiRo}*{Chapter 3} for an overview of the SG-calculus):
\begin{enumerate}
    \item $\displaystyle\calcsg(\RR^d)=\bigcup_{(m_\psi,m_e)\in\RR^2}\calcsg^{m_\psi,m_e}(\RR^d)$ is a graded *-algebra;
    	its elements are linear continuous operators from $\Sw(\RR^d)$ to itself, extendable to linear
        continuous operators on $\Sw'(\RR^d)$;
    \item the differential operators of the form
        \begin{align*}
            \sum_{|\alpha| \leq m_e, |\beta| \leq m_\psi} a_{\alpha,\beta} x^\alpha D^\beta, \quad m_e,m_\psi\in\NN,
        \end{align*}
        are $\SG$ operators of order $(m_\psi,m_e)$;
    \item If $A \in \calcsg^{0,0}(\RR^d)$, then $A$ extends to a bounded linear operator
        \begin{align*}
            A : L^2(\RR^d) \to L^2(\RR^d);
        \end{align*}
    \item there is an associated scale of $\SG$-Sobolev spaces (also known as Sobolev-Kato spaces), defined by
        \begin{align*}
            H^{s_\psi,s_e}(\RR^d) = \{u \in \Sw'(\RR^d) \colon \| \ang{\cdot}^{s_e} \ang{D}^{s_\psi} u\|_{L^2(\RR^d)} < \infty \},
        \end{align*}
        and, for all $m_\psi,m_e,s_\psi,s_e\in\RR$, the operator $A \in \calcsg^{m_\psi,m_e}(\RR^d)$ is a bounded linear operator
        \begin{align*}
            A : H^{s_\psi,s_e}(\RR^d) \to H^{s_\psi-m_\psi,s_e-m_e}(\RR^d);
        \end{align*}
    \item the inclusions $H^{s_\psi,s_e}(\RR^d)\subset H^{r_\psi,r_e}(\RR^d)$, $s_\psi\ge r_\psi$, $s_e\ge r_e$,
    are continuous, compact when the order components inequalities are both strict; 
    moreover, the scale of the Sobolev-Kato spaces is global in the sense that
        \begin{align*}
            \bigcup_{s_\psi,s_e} H^{s_\psi,s_e}(\RR^d) = \Sw'(\RR^d), \quad \bigcap_{s_\psi,s_e} H^{s_\psi,s_e}(\RR^d) = \Sw(\RR^d);
        \end{align*}
    \item an operator $A = \Op^w(a) \in \calcsg^{m_\psi,m_e}(\RR^d)$ is elliptic if its symbol $a$ is invertible for $|x|+|\xi| \ge R>0$, 
        and $\chi(|x|+|\xi|)[a(x,\xi)]^{-1}$ is a symbol in $\symbsg^{-m_\psi,-m_e}(\RR^{2d})$, where
        $\chi \in \Sm(\RR)$ with $\chi(x) = 1$ for $x > 2R$ and $\chi(x) = 0$ for $x < R$;
  	
        \item if $A\in \calcsg^{m_\psi,m_e}(\RR^d)$ is an elliptic operator, then there is a parametrix $B \in \calcsg^{-m_\psi,-m_e}(\RR^d)$ such that
        \begin{align*}
            AB - \id \in \calcsg^{-\infty,-\infty}(\RR^d), \quad BA - \id \in \calcsg^{-\infty,-\infty}(\RR^d).
        \end{align*}
\end{enumerate}
\subsection{SG-Classical Symbols}\label{subs:sgcl}
We first introduce two classes of $\SG$-symbols which are \textit{homogeneous in the large} with respect either to the
variable or the covariable. For any $\rho>0$, $x_0\in\RR^d$, we let $B_\rho(x_0)=\{x\in\RR^d\colon |x-x_0|<\rho \}$
and we fix a cut-off function $\omega \in \Smc(\RR^d)$ with $\omega \equiv 1$ on the ball $B_\frac{1}{2}(0)$.
\begin{enumerate}[(1)]
    \item A symbol $a = a(x, \xi)$ belongs to the class $\SG^{m_\psi,m_e}_{\cl(\xi)}(\RR^{2d})$ if there exist functions $a_{m_\psi-i, \cdot} (x, \xi)$, $i=0,1,\dots$, homogeneous of degree $m_\psi-i$ with respect to the variable
    $\xi$, smooth with respect to the variable $x$, such that,
    \[
    a(x, \xi) - \sum_{i=0}^{M-1} (1-\omega(\xi)) \, a_{m_\psi-i, \cdot} (x, \xi)\in \SG^{m_\psi-M, m_e}(\RR^{2d}), \quad M=1,2, \ldots
    \]
    \item A symbol $a$ belongs to the class $\SG_{\cl(x)}^{m_\psi,m_e}(\RR^{2d})$ if $a \circ R \in \SG^{m_e,m_\psi}_{\cl(\xi)}(\RR^{2d})$, where $R(x,\xi) = (\xi, x)$.
        This means that $a(x,\xi)$ has an asymptotic expansion into homogeneous terms in $x$.
\end{enumerate}
\begin{definition}\label{def:sgcl}
A symbol $a$ is called $\SG$-classical, and we write $a \in \SG_{\cl(x,\xi)}^{m_\psi,m_e}(\RR^{2d})=\SG_{\cl}^{m_\psi,m_e}(\RR^{2d})$, 
if the following two conditions hold true:
\begin{enumerate}[(i)]
    \item there exist functions $a_{m_\psi-j, \cdot} (x, \xi)$, homogeneous of degree $m_\psi-j$ with respect to $\xi$ and smooth in $x$, such that
    $(1-\omega(\xi)) a_{m_\psi-j, \cdot} (x, \xi)\in \SG_{\cl(x)}^{m_\psi-j, m_e}(\RR^{2d})$ and
    \[
    a(x, \xi)- \sum_{j=0}^{M-1} (1-\omega(\xi)) \, a_{m_\psi-j, \cdot}(x, \xi) \in \SG^{m_\psi-M, m_e}_{\cl(x)}(\RR^{2d}), \quad M=1,2,\dots;
    \]
    \item there exist functions $a_{\cdot, m_e-k}(x, \xi)$, homogeneous of degree $m_e-k$ with respect to the $x$ and smooth in $\xi$, such that
    $(1-\omega(x))a_{\cdot, m_e-k}(x, \xi)\in \SG_{\cl(\xi)}^{m_\psi, m_e-k}(\RR^{2d})$ and
    \[
    a(x, \xi) - \sum_{k=0}^{M-1} (1-\omega(x)) \, a_{\cdot, m_e-k} (x,\xi)\in \SG^{m_\psi, m_e-M}_{\cl(\xi)}(\RR^{2d}), \quad M=1,2,\dots
    \] 
\end{enumerate}
\end{definition}
Note that the definition of $\SG$-classical symbol implies a condition of compatibility for the terms of the expansions with respect to $x$ and $\xi$. In fact, defining $\sigma^\psi_{m_\psi-j}$ and $\sigma^e_{m_e-i}$ on $\SG_{\cl(\xi)}^{m_\psi,m_e}$ and $\SG_{\cl(x)}^{m_\psi,m_e}$, respectively,  as
\begin{align*}
	\sigma^\psi_{m_\psi-j}(a)(x, \xi) &= a_{m_\psi-j, \cdot}(x, \xi),\quad j=0, 1, \ldots, 
	\\
	\sigma^e_{m_e-k}(a)(x, \xi) &= a_{\cdot, m_e-k}(x, \xi),\quad k=0, 1, \ldots,
\end{align*}
it possible to prove that 
\[
\begin{split}
a_{m_\psi-j,m_e-k}=\sigma^{\psi e}_{m_\psi-j,m_e-k}(a)=\sigma^\psi_{m_\psi-j}(\sigma^e_{m_e-k}(a))= \sigma^e_{m_e-k}(\sigma^\psi_{m_\psi-j}(a))
\end{split}
\]
for all $j,k \in \NN$.

Moreover, the composition of two $\SG$-classical operators is still classical. 
For $A=\Op{a}\in \calcclsg^{m_\psi,m_e}(\RR^d)$ 
the triple 
\begin{align*}
    \sigma(A) = (\sigma^\psi(A),\sigma^e(A),\sigma^{\psi e}(A))=(a_\psi,a_e,a_\psie).
\end{align*}
where
\begin{align*}
    \sigma^\psi(A)(x,\xi) = a_\psi(x,\xi) &= a_{m_\psi,\cdot}\left(x,\frac{\xi}{|\xi|}\right),\\
    \sigma^e(A)(x,\xi) = a_e(x,\xi) &= a_{\cdot,m_e}\left(\frac{x}{|x|},\xi\right),\\
    \sigma^{\psi,e}(A)(x,\xi) = a_\psie(x,\xi) &= a_{m_\psi,m_e}\left(\frac{x}{|x|},\frac{\xi}{|\xi|}\right)
\end{align*}
is called the \emph{principal symbol} of $A$. 
This definition keeps the usual multiplicative behaviour, that is, for any $A\in \calcclsg^{m_\psi,m_e}(\RR^d)$, 
$B\in \calcclsg^{r_\psi,r_e}(\RR^d)$, $(m_\psi,m_e),(r_\psi,r_e)\in\RR^2$,
the principal symbol of $AB$ is given by
\[\sigma(AB) = \sigma(A) \cdot \sigma(B),\]
where the product is taken component-wise.
Proposition~\ref{prop:ellclass} below allows to express the ellipticity
of $\SG$-classical operators in terms of their principal symbol.
Fixing a cut-off function $\omega \in \Smc(\RR^d)$ as above, we define the principal part of $a$ to be
\begin{equation}\label{eq:ap}
    a_p(x,\xi) = (1 - \omega(\xi))a_\psi(x,\xi) + (1 - \omega(x)) (a_e(x,\xi) - (1 - \omega(\xi))a_\psie(x,\xi)).
\end{equation}
\subsection{SG-wavefront sets}\label{subs:wfs}
We denote by $\Wt$ the disjoint union
\begin{align*}
    \Wt = \Wp\sqcup\We\sqcup\Wpe=(\RR^d \times \Sph^{d-1}) \sqcup (\Sph^{d-1}\times \RR^d) \sqcup (\Sph^{d-1} \times \Sph^{d-1}),
\end{align*}
which may be viewed as the boundary of the (double) radial compactification of the phase space 
$T^*\RR^d\simeq\RR^{d}\times\RR^d$.
Therefore, it is natural to define smooth functions on $\Wt$ as follows:
\begin{align*}
    \Sm(\Wt) = \{ &(f_\psi,f_e,f_\psie) \in \Sm(\Wp)\times \Sm(\We) \times \Sm(\Wpe) \colon\\
    &\lim_{\lambda\to \infty} f_\psi(\lambda x, \xi) = \lim_{\lambda\to \infty} f_e(x,\lambda \xi) = f_\psie(x,\xi) \text{ for all } (x,\xi) \in \Sph^{d-1}\times \Sph^{d-1}\}.
\end{align*}
By restriction, the principal symbol can be defined as a map $\sigma : \SG_\cl^{m_\psi,m_e}(\RR^{2d})\ni a\mapsto \sigma(a) \in \Sm(\Wt)$.
\begin{proposition}\label{prop:ellclass}
	An operator $A\in \calcclsg^{m_\psi,m_e}(\RR^d)$ is elliptic if and only if
    $\sigma(A)(x,\xi) \not = 0$ for all $(x,\xi) \in \Wt$.
\end{proposition}

For $A \in \SG_\cl^{m_\psi,m_e}(\RR^d)$ we define the following sets (see \cites{CoMa13,MelroseSST}):
\begin{enumerate}
    \item the elliptic set
        \begin{align*}
            \Ell(A) = \{ (x,\xi) \in \Wt \colon \sigma(A)(x,\xi) \not = 0\},
        \end{align*}
    \item the characteristic set
        \begin{align*}
            \Char(A) = \Wt \setminus \Ell(A),
        \end{align*}
    \item the operator $\SG$-wavefront set $\WFSG'(A)$, defined by
        $(x,\xi) \notin \WFSG'(A)$ if there exists 
        $B \in \Op\SG^{0,0}_\cl(\RR^d)$ such that $AB \in \Op\SG^{-\infty,-\infty}(\RR^d)$ satisfying $(x,\xi) \in \Ell(B)$, 
        or, more concisely,
        \begin{align*}
            \WFSG'(A) = \bigcap_{\substack{B \in \Op\SG_\cl^{0,0}\\AB \in \Op\SG^{-\infty,-\infty}}} \Char(B).
        \end{align*}
\end{enumerate}

The $\SG$-wavefront set of a distribution $u \in \Sw'(\RR^d)$ is defined as
\begin{align*}
    \WFSG(u) = \bigcap_{\substack{A \in \Op\SG_\cl^{0,0}\\Au \in \Sw(\RR^d)}} \Char(A),
\end{align*}
see \cites{Cordes, CoMa13, MelroseSST}. We will decompose the $\SG$-wavefront set of $u\in\Sw'(\RR^d)$
into its components in $\Wt$, namely,
\begin{align*}
    \WFSG(u) = (\WFSG^\psi(u), \WFSG^e(u), \WFSG^\psie(u)), 
    \quad
    \WFSG^\bullet(u) \subset \Wt^\bullet, \bullet\in\{\psi,e,\psie\}.
\end{align*}
Then, we have that
\begin{align*}
    \WFSG^\psi(u) = \WF_\cl(u),
\end{align*}
where $\WF_\cl(u)$ is the classical H\"ormander's wavefront set.

The $\SG$-wavefront set is well-behaved with respect to the Fourier transform (see, e.g.,  \cite{CoMa03}*{Lemma 2.4}):
\begin{align*}
    (x,\xi) \in \WFSG(u)\Longleftrightarrow (\xi,-x)\in\WFSG(\widehat{u}).
\end{align*}
\subsection{Complex Powers}
As in the case of closed manifolds, it is possible to define complex powers of $\SG$-pseudodifferential operators.
We will only review the crucial properties of complex powers for a positive elliptic 
self-adjoint operator $A \in \calcclsg^{m_\psi,m_e}(\RR^d)$,
$m_\psi,m_e>0$. For the definition and proofs of the following properties, we refer to \cite{BaCo11} (cf. also \cites{MaScSe06, Schrohe88}).
\begin{enumerate}[(i)]
    \item $A^z A^s= A^{z+s}$ for all $z, s \in \CC$.
    \item $A^k= \underbrace{A \circ \ldots \circ A}_{k \text{ times}}$ for $k \in \NN$.
    \item If $A \in \calcclsg^{m_\psi, m_e}(\RR^d)$, then $A^z \in \calcclsg^{m_\psi \Re z, m_e \Re z}(\RR^d)$.
    \item If $A$ is a classical $\SG$-operator, then $A^z$ is classical and its principal symbol is given by
        \begin{align*}
            \sigma(A^z) = \sigma(A)^z.
        \end{align*}
    \item\label{it:traceclass} For $\Re z < -d \cdot \min\{1/m_e, 1/m_\psi\}$, $A^z$ is trace-class.
\end{enumerate}

For any $A=\Op(a)\in\calcclsg^{m_\psi, m_e}(\RR^d)$ as above,
the full symbol\footnote{For the definition of the zeta function it does not matter which quantization we choose.}
of $A^z$ will be denoted by \[a(z) \in\calcclsg^{m_\psi \Re z, m_e \Re z}(\RR^d).\]%

Let $s \in \CC$ with $\Re(s) > \max\{d/m_e, d/m_\psi\}$.
Using the property~\eqref{it:traceclass} it is possible to define $\zeta(s)$ by
\begin{equation}\label{eq:zetaf}
    \zeta(s) = \Tr A^{-s} = \int K_{A^{-s}}(x, x) dx = (2\pi)^{-d} \iint a(x,\xi; -s) dx\,d\xi,
\end{equation}
where $K_{A^z}$ is the Schwartz kernel of $A^z$.
We note that the $\zeta$-function may be written as
\begin{align*}
    \zeta(s) = \sum_{j=1}^\infty \lambda_j^{-s}.
\end{align*}
with $(\lambda_j)_{j\in\NN}$ the sequence of eigenvalues of $A$.

\begin{theorem}[Battisti--Coriasco~\cite{BaCo11}]\label{thm:fzeta}
    The function $\zeta(s)$ is holomorphic for $\Re(s)>d\cdot\max \{1/m_e, 1/m_e\}$.
    Moreover, it can be extended as a meromorphic function with possible poles at the points
    \[
    s_j^1=\frac{d-j}{m_\psi}, \, j=0, 1, \ldots, \quad s^2_k=\frac{d-k}{m_e}, \, k=0, 1, \ldots
    \]
    Such poles can be of order two if and only if there exist integers $j, k$ such that
    \begin{equation}
    \label{conpoli}
    s_j^1=\frac{d-j}{m_\psi}=\frac{d-k}{m_e}=s_k^2.
    \end{equation}
\end{theorem} 
\subsection{Parametrix of SG-hyperbolic Cauchy problems}\label{subs:waveq}
Let $P \in \calcclsg^{1,1}(\RR^d)$ be a self-adjoint positive elliptic operator.
By the construction from \cite{CoPa02}*{Theorem 1.2} (cf. also \cites{CoriascoFIO1,CoriascoFIO2,CoMa03}), it is 
possible to calculate a suitable parametrix for the Cauchy problem associated with the 
wave equation, namely,
\begin{equation}\label{eq:wave}
    \left\{\begin{aligned}
        (i\pa_t - P)u(t,x) &= 0\\
        u(0,x) &= u_0(x).
    \end{aligned}\right.
\end{equation}

The solution operator of \eqref{eq:wave} exists by the spectral theorem and
is denoted by $U(t) = e^{-itP} = [\F_{\lambda\to t}(dE)](t)$, where $dE$ is the spectral measure of $P$.
There exists a short time parametrix $\widetilde{U}(t)$, which is 
given by operators defined through the integral kernels
\begin{equation}\label{eq:utker}
    K_{\widetilde{U}(t)}(x,y) = (2\pi)^{-d} \int e^{i(\phi(t,x,\xi) - y\xi)} \tilde{a}(t,x,\xi) d\xi,
\end{equation}
where $\tilde{a} \in \Smc( (-2\ep,2\ep), \symbclsg^{0,0})$ with $\tilde{a}(0) - 1 \in \symbsg^{-\infty,-\infty}$ and $\phi \in \Sm( (-\ep,\ep), \symbclsg^{1,1})$.

The parametrix $\widetilde{U}(t)$ solves the wave equation \eqref{eq:wave}
in the sense that $\tilde{u}(t,x) = [\widetilde{U}(t) u_0](x)$ satisfies
\begin{equation}\label{eq:wave-parametrix}
    \left\{\begin{aligned}
        (i\pa_t - P)\tilde{u}(t) &\in \Sm( (-\ep/2,\ep/2), \Sw(\RR^d)) \\
        \tilde{u}(0) - u_0 &\in \Sw(\RR^d).
    \end{aligned}\right.
\end{equation}
By a Duhamel argument, $U(t)-\widetilde{U}(t) \in \Sm( (-\ep, \ep), \mathcal{L}(\Sw'(\RR^d), \Sw(\RR^d)))$,
(cf. \cite{CoMa13}*{Theorem 16}, \cite{DGW18}*{p.~284}).
Since the error term is regularizing, we obtain that
\begin{align}
    K_{U(t)}(x,y) = (2\pi)^{-d}\int e^{i(\phi(t,x,\xi) - y\xi)} a(t,x,\xi) d\xi,
\end{align}
for $a \in \Smc( (-\ep,\ep), \symbclsg^{0,0})$ with $a(0) = 1$ (cf. \cite{CDS19}*{Lemma~4.14}).

Let $p$ be the principal part of the full Weyl-quantized symbol of $P$.
The phase function $\phi$ satisfies the eikonal equation
\begin{equation}\label{eq:eikonal}
    \left\{\begin{aligned}
        \pa_t \phi(t,x,\xi) + p(x,\phi^\prime_x(t,x,\xi)) &= 0\\
        \phi(0,x,\xi) &= x\xi.
    \end{aligned}\right.
\end{equation}
This implies that we have a Taylor expansion in $t$ of the form
\begin{equation}\label{eq:phit}
    \phi(t,x,\xi) = x\xi - t p(x,\xi) + t^2 \Sm(\RR_t, \SG^{1,1}_\cl)
\end{equation}
for $t$ small enough.

For any $f \in \Sm(\RR^{2d})$, we define the Hamiltonian vector field by
\begin{align*}
    \hamvf_f = \ang{\pa_x f, \pa_\xi} - \ang{\pa_\xi f, \pa_x}
\end{align*}
and we denote its flow by $t \mapsto \exp(t\hamvf_f)$. For $P\in\calcclsg^{1,1}(\RR^d)$, we will collectively 
denote by $\hamvf_{\sigma(P)}$ the Hamiltonian 
vector fields on $\Wt^\bullet$ generated by $\sigma^\bullet(P)$, $\bullet\in\{\psi,e,\psie\}$,
and by $t \mapsto \exp(t\hamvf_{\sigma(P)})$ the three corresponding flows.

By the group property, $U(t+s) = U(t) U(s)$, we can extend propagation of singularities results for small times to $t \in \RR$.
In \cite{CoMa03} the propagation of the $\SG$-wavefront set under the action of $\SG$-classical operators and operator
families like $U(t)$ has been studied. In particular, the following theorem was proved there 
(see also \cite{CJT4}).
\begin{theorem}\label{thm:sgwfsprop}
    Let $u_0\in\Sw'(\RR^d)$ and $U(t) = e^{-itP}$.
	Then, \[\WFSG^\bullet(U(t)u_0)\subseteq \Phi^\bullet(t)(\WFSG^\bullet(u_0)),\] where
	$\Phi^\bullet$ is the smooth family of canonical transformations on $\Wt^\bullet$ generated by $\sigma^\bullet(\phi)$ with
	$\bullet\in\{\psi,e,\psie\}$. 
\end{theorem}
\begin{remark}
	In view of \eqref{eq:phit}, Theorem \ref{thm:sgwfsprop} can also be stated in the following way: for any $u_0\in\Sw'(\RR^d)$
	and $t\in(-\ep/2,\ep/2)$, $\WFSG^\bullet(U(t)u_0)\subset\exp(t\hamvf_{\sigma^\bullet(p)})(\WFSG^\bullet(u_0))$, where
	$\bullet\in\{\psi,e,\psie\}$, and $\hamvf_f$ is the Hamiltonian vector field generated by $f$. 
	In the sequel we will express this fact in the compact form
	\[
		\WFSG(U(t)u_0)\subset\exp(t\hamvf_{\sigma(p)})(\WFSG(u_0)), \quad u_0\in\Sw'(\RR^d), t\in \RR.
	\]
\end{remark}

\section{Wave Trace}\label{sec:wt}
We fix a positive elliptic operator $P \in \Op\SG_\cl^{1,1}(\RR^{2d})$ 
with $\psie$-principal symbol $p_\psie = \sigma^\psie(P)$.
By the compactness of the embedding of $\SG$-Sobolev spaces, we have that the resolvent $(\lambda - P)^{-1}$
is compact for $\lambda > 0$ and hence
there exists an orthonormal basis $\{\psi_j\}$ of $L^2$ consisting 
of eigenfunctions of $P$ with eigenvalues $\lambda_j$ with the property that
\begin{align*}
    0 < \lambda_1 \leq \lambda_2 \leq \dots \to + \infty.
\end{align*}
Therefore, the spectral measure is given by
$dE(\lambda) = \sum_{j=1}^\infty \delta_{\lambda_j}(\lambda)\ang{\cdot, \psi_j} \psi_j$, where $\delta_\mu$ is the delta distribution centered at $\mu$, and we have that
\begin{align*}
    N(\lambda) = \Tr \int_0^\lambda dE(\lambda).
\end{align*}

The \emph{wave trace} $w(t)$ is (formally) defined as
\begin{align*}
    w(t) = \Tr U(t) = \sum_{j=1}^\infty e^{-it\lambda_j}.
\end{align*}
As usual, $w(t)$ is well-defined as a distribution by means of integration by parts and the fact that $P^{-N}$ is trace-class for $N > d$ (cf. Schrohe~\cite{Schrohe88}*{Theorem 2.4}).

Theorem~\ref{thm:sgwfsprop} directly implies the following Lemma \ref{lem:prop-sing}.
\begin{lemma}\label{lem:prop-sing}
    Choose $t_0 \in \RR$.
    Let $\Gamma \subset \Wt$ be open and such that $\left[\exp(t\hamvf_{\sigma(P)}) (\Gamma)\right] \cap \Gamma = \emptyset$,
    for all $t\in(t_0-\delta,t_0+\delta)$ and $\delta > 0$ small.
    Then, for all $B \in \Op\SG_\cl^{0,0}(\RR^{2d})$ with $\WFSG'(B) \subset \Gamma$, and all $t\in(t_0-\delta,t_0+\delta)$,
    we have that $B U(t) B \in \mathcal{L}(\Sw'(\RR^d), \Sw(\RR^d))$.
\end{lemma}
We will show that the improvement of the Weyl law is only related to the corner component
\begin{align*}
    \{ t \in \RR \colon \exp(t\hamvf_{\sigma^\psie(P)})(x,\xi) = (x,\xi) \text{ for some } (x,\xi) \in \Wt^\psie\}.
\end{align*}

The structure of the singularities of $w(t)$ is more involved. This comes from the fact that the boundary at infinity is not a manifold or equivalently the flow is not homogeneous.
In contrast to the case of a closed manifold, the distribution $w(t)$ will not be a conormal distribution near $0$,
but it turns out that it is a log-polyhomogeneous distribution.

Let $\ep > 0$ as in Section~\ref{subs:waveq} and choose a function $\chi \in \Sw(\RR)$ with $\supp \hat\chi \subset (-\ep, \ep)$ and $\hat\chi = 1$ on $(-\ep/2,\ep/2)$.
\begin{proposition}\label{prop:micro-wavetrace}
    Let $B \in \Op\SG_\cl^{0,0}$ and denote by $N_B(\lambda) = \Tr (E_\lambda B B^*)$ the microlocalized counting function.
    There exist coefficients $w_{jk} \in \RR$ with $k \in \NN$ and $j\in \{0,1\}$ independent of $\chi$ such that
    \begin{equation}\label{eq:microlocN}
        (N_B * \chi)(\lambda) \sim \sum_{k=0}^\infty \sum_{j=0,1} w_{jk} \lambda^{d-k} (\log\lambda)^j
    \end{equation}
    as $\lambda \to \infty$.
\end{proposition}
\begin{remark}\label{rem:FTNB}
	Note that $[\F(N_B')](t) = \Tr (U(t) B B^*)$.
\end{remark}
\begin{proof}
    From Section~\ref{subs:waveq}, we obtain that there is a parametrix $\widetilde{U}(t)$ for $U(t)$ and we have
    \begin{align*}
        K_{U(t) B B^*}(x,y) = (2\pi)^{-d} \int e^{i(\phi(t,x,\xi) - y\xi)} a(t,x,\xi) \,d\xi
    \end{align*}
    for $t \in (-\ep,\ep)$.
    The amplitude satisfies $\sigma(a(0)) = \sigma(BB^*)$.

    Set
    \begin{align*}
        \T_B(t) = \hat\chi(t) \Tr (U(t) B B^*).
    \end{align*}
    By the previous remark, we have that $\T_B(t)$ is the Fourier transform of $(N_B' * \chi)(\lambda)$. We will now calculate the inverse Fourier transform of $\T_B$.

    Using the Taylor expansion of the phase function, we have that
    \begin{align*}
        \phi(t,x,\xi) = x\xi + t\psi(t,x,\xi),
    \end{align*}
    where $\psi$ is smooth in $t$.
    Formally, we can write the trace as
    \begin{align*}
        \T_B(t) = (2\pi)^{-d} \hat\chi(t)\int e^{it\psi(t,x,\xi)} a(t,x,\xi) \,dx\,d\xi.
    \end{align*}
    As in Hörmander~\cite{Hormander4} we set
    \begin{align}\label{eq:tilde-A}
        \tilde{A}_B(t,\lambda) = (2\pi)^{-d} \hat\chi(t) \int_{\{-\psi(t,x,\xi) \leq \lambda\}} a(t,x,\xi) dx\,d\xi.
    \end{align}
    Note that ellipicity implies that $\tilde{A}_B(t,\lambda) < \infty$.
    By the Push-Forward Theorem (cf. Melrose~\cite{Melrose92} and Grieser and Gruber~\cite{GrGr99}) 
    it follows from \eqref{eq:tilde-A} that
    $\partial_\lambda \tilde{A}_B(t,\lambda)$ is log-homogeneous of order $d-1$. 
    Defining $A_B(\lambda) = e^{iD_tD_\lambda}\tilde{A}_B(t,\lambda)|_{t=0}$, we find
    \begin{align*}
        \T_B(t) = \int_{\RR} e^{-it\lambda} \partial_\lambda A_B(\lambda) d\lambda.
    \end{align*}

    The above implies that $\tilde{A}_B$ and $A_B$ are log-homogeneous of order $d$.
    In particular, we have that
    \begin{align*}
        A_B(\lambda) = \sum_{k=0}^\infty \sum_{j=0,1} w_{jk}\lambda^{d-k}  (\log\lambda)^j + O(\lambda^{-\infty}).
    \end{align*}

    We conclude that
    \begin{align*}
        (N_B * \chi)(\lambda) &=
        \int_{-\infty}^\lambda\F^{-1}_{t\to \lambda}\{\T_B\}(\lambda) \,d\lambda\\
        &= A_B(\lambda)\\
        &= \sum_{k=0}^\infty \sum_{j=0,1} w_{jk}\lambda^{d-k}  (\log\lambda)^j + O(\lambda^{-\infty}).
    \end{align*}
    We note that the coefficients are determined by derivatives of $\tilde{A}_B(t,\lambda)$ at $t = 0$ and since $\hat\chi = 1$ near $t=0$, the specific choice of $\chi$ does not change the coefficients.
\end{proof}

\section{Relation with the spectral \texorpdfstring{$\zeta$-function}{zeta-function}}\label{sec:zfandwt}
As in the case of pseudodifferential operators on closed manifolds (cf. Duistermaat and Guillemin~\cite{DuGu75}*{Corollary 2.2}),
the wave trace at $t = 0$ is related to the spectral $\zeta$-function. This relation extends to the $\SG$ setting.

Recall that for a positive self-adjoint elliptic operator $P \in \Op\SG^{1,1}_\cl(\RR^d)$, the function
$\zeta(s)$ is defined for $\Re s > d$ by
\begin{align*}
    \zeta(s) = \Tr P^{-s}.
\end{align*}
In addition, we consider the microlocalized version of $\zeta(s)$, defined by
\begin{align*}
    \zeta_B(s) = \Tr( P^{-s} B B^*) = \sum_{j=1}^\infty \lambda_j^{-s} \|B^* \psi_j\|^2,\quad \Re s > d,
\end{align*}
for $B \in \Op\SG_\cl^{0,0}$. Of course, $\zeta_{\id}(s) = \zeta(s)$.

By Theorem~\ref{thm:fzeta}, $\zeta(s)$ admits a meromorphic continuation to $\CC$ with poles of maximal order two at $d-k$, 
$k\in\NN$. This result extends to $\zeta_B(s)$ and we characterize the Laurent coefficients in terms of the wave
trace expansion at $t = 0$.

\begin{proposition}\label{prop:zeta}
    The function $\zeta_B(s)$ extends meromorphically to $\CC$ and has at most poles of order two at the points $d - k$, 
    $k\in\NN$. We have the expansion
    \begin{align*}
        \zeta_B(s) = \frac{A_{2,k}}{[s-(d-k)]^2} + \frac{A_{1,k}}{s-(d-k)} + f(s),
    \end{align*}
    where $f$ is holomorphic near $s = d-k$ and
    \begin{equation}\label{eq:atow}
    \begin{aligned}
        A_{2,k} &= (d-k) w_{1k},\\
        A_{1,k} &= w_{1k} + (d-k) w_{0k},
    \end{aligned}
    \end{equation}
    where the $w_{jk}$, $k\in\NN$, $j=0,1$, are the coefficients appearing in the asymptotic expansion \eqref{eq:microlocN}
    of $N_B(\lambda)$.
\end{proposition}

\begin{proof}
The meromorphic continuation and the possible location of the poles follow from similar arguments as in \cite{BaCo11}*{Theorem 3.2}
(see also the proof of Proposition \ref{prop:third-term}).
Hence, we only have to show that the poles are related to $N_B(\lambda)$.

Let $\ep \in(0, \lambda_1)$ be sufficiently small.
Choose an excision function $\chi \in \Sm(\RR)$ such that $\chi(\lambda) = 0$ for $\lambda < \ep$ and $\chi(\lambda) = 1$ for $\lambda \geq \lambda_1$.
Set $\chi_s(\lambda) = \chi(\lambda) \lambda^{-s}$.
Then, using Remark \ref{rem:FTNB},
\begin{align*}
    \zeta_B(s) = \ang{N_B', \chi_s} = \ang{\Tr (U(t)BB^*), \F^{-1}(\chi_s)}.
\end{align*}
Let $\rho \in \Sw(\RR)$ such that $\rho$ is positive, $\hat\rho(0) = 1$, $\hat \rho \in \Smc(\RR)$, and $\rho$ is even.
By an argument similar to the one in \cite{DuGu75}*{Corollary 2.2}, we have that
\begin{align*}
    \zeta_B(s) - \ang{N_B' * \rho, \chi_s} = \ang{(1 - \hat\rho) \Tr (U(t)BB^*), \F^{-1}(\chi_s)}
\end{align*}
is entire in $s$ and polynomially bounded for $\Re s > C$.

Now, we can insert the asymptotic expansion of $N_B' * \rho$ to calculate the residues of $\zeta_B(s)$.
Taking the derivative of \eqref{eq:microlocN}, we see that
the asymptotic expansion of $N_B' * \rho$ is given by
\begin{equation}\label{eq:microlocNp}
    (N_B' * \rho)(\lambda) = \sum_{k=0}^N\sum_{j=0,1}  A_{j+1,k} \lambda^{d-k-1} (\log\lambda)^j + o(\lambda^{d-1-N})
\end{equation}
for any $N \in \NN$ and $A_{j,k}$ are given by \eqref{eq:atow}.

Let $k\in\NN$ be arbitrary.
If $f \in \Sm(\RR)$ with $f(\lambda) = O(\lambda^{d-k-1}\log\lambda)$ as $\lambda \to \infty$, then 
$\int f(\lambda) \chi(\lambda) \lambda^{-s} d\lambda$ is bounded
and holomorphic in $s$ for $\Re s > d - k$.
Let
\begin{align*}
    I(s) = \int \lambda^{d-k-1} \chi(\lambda) \lambda^{-s} d\lambda.
\end{align*}
By partial integration we obtain
\begin{align*}
    I(s) &= \frac{\psi(s)}{s-(d-k)}.
\end{align*}
where $\psi(s) = \int \lambda^{d-k-s} \chi'(\lambda) d\lambda$ is holomorphic and $\psi(d-k) = 1$.
Therefore, we have
\begin{align*}
    \int \lambda^{d-k-s-1} (A_{1,k}+A_{2,k}\log\lambda) \chi(\lambda) d\lambda = -A_{2,k} I'(s) + A_{1,k} I(s).
\end{align*}
Hence, the integral near $s = d-k$ is given by
\begin{align*}
    \int \lambda^{d-k-s-1} (A_{1,k}+A_{2,k}\log\lambda) \chi(\lambda) d\lambda = \frac{A_{2,k}}{[s-(d-k)]^2} + \frac{A_{1,k}}{s-(d-k)} + f(s),
\end{align*}
where $f$ is holomorphic in a neighbourhood of $s = d-k$. The formulae relating the coefficients $A_{j+1,k}$ and  $w_{jk}$, $j=0,1$,
are obtained by comparing the $\lambda$-derivative of \eqref{eq:microlocN} with \eqref{eq:microlocNp}.
\end{proof}

The main advantage in employing the $\zeta$-function is that the coefficients are easier to calculate than for the wave trace.
\begin{proposition}\label{first-term}
    Let $B \in \Op\SG_\cl^{0,0}$ with principal $\psie$-symbol $b_\psie$. 
    The function $\zeta_B(s)$ has a pole of order two at $s = d$ with leading Laurent coefficient
    \begin{align*}
        (2\pi)^{-d} \int_{\Sph^{d-1}}\int_{\Sph^{d-1}} [p_\psie(\theta,\omega)]^{-d} \cdot b_\psie(\theta,\omega) d\theta \,d\omega.
    \end{align*}
\end{proposition}
\begin{proof}
    This follows from the same arguments as in \cite{BaCo11} (cf. the proof of Proposition \ref{prop:third-term} below),
    with the modification that the full symbol is $a(z) = p(z) \# b$, where $p(z)$ denotes the full symbol of $P^z$.
    The principal $\psie$-symbol of $A(z) = P^z B$ is given by $a_{z,z}(x,\xi;z) = [p_\psie(x,\xi)]^z \cdot b_\psie(x,\xi)$.
\end{proof}
For the three-term asymptotics, we compute the third coefficient more explicitly.
\begin{proposition}\label{prop:third-term}
    Let $p(s) = p(x,\xi; s)$ be the full symbol of $P^s$.
    The leading Laurent coefficient of $\zeta(s)$ at $s = d - 1$ is given by
    \begin{align*}
        (2\pi)^{-d} \int_{\Sph^{d-1}}\int_{\Sph^{d-1}} p_{-d,-d}(\theta,\omega;-d+1) d\theta\,d\omega.
    \end{align*}
\end{proposition}
\begin{proof}
By the analysis performed in \cite{BaCo11}, it follows that 
\[
    \zeta(s)=\sum_{j=1}^4\zeta_j(s),
\]
where, for $\Re s>d$, 
\[
    \zeta_j(s)=(2\pi)^{-d}\int_{\Omega_j}p(x,\xi;-s)dxd\xi
\]
and
\begin{align*}
    \Omega_1&=\{(x, \xi)\colon |x|\leq1, |\xi|\leq1\},\quad \Omega_2=\{(x, \xi)\colon |x|\le1, |\xi|> 1\},\\
\Omega_3&=\{(x, \xi)\colon |x|>1, |\xi|\le1\},\quad \Omega_4=\{(x, \xi)\colon |x|>1, |\xi|>1\}.
\end{align*}
Let us recall the main aspects of the proof of the properties of the four terms $\zeta_j(s)$, $j=1,\dots,4$, showed in \cite{BaCo11}.
\begin{enumerate}[(1)]
	\item $\zeta_1(s)$ is holomorphic, since we integrate $p(-s)$, a holomorphic function in $s$ and smooth 
	with respect to $(x,\xi)$, on a bounded set with respect to $(x,\xi)$.

    \item\label{it:zeta2} Let us first assume $\Re s > d$. Using the expansion of $p(-s)$ with $M\ge1$ terms homogeneous 
	with respect to $\xi$, switching to polar coordinates in $\xi$ and integrating the radial 
	part, one can write
	\begin{align*}
	\zeta_2(s)&=(2\pi)^{-d} 
	\sum_{j=0}^{M-1}\frac{1}{s-(d-j)} \int_{|x|\le 1}\int_{\mathbb{S}^{d-1}} p_{-s-j, \cdot} (x, \omega;-s) d\omega dx
	\\
	&+(2\pi)^{-d}\iint_{\Omega_2} r_{-s-M, \cdot}(x, \xi; -s) d\xi dx.
	\end{align*}
	Notice that the last integral is convergent, and provides a holomorphic function in $s$. 
	Arguing similarly to the case of operators on smooth, compact manifolds, 
	$\zeta_2(s)$ turns out to be holomorphic for $\Re(s)>d$, extendable 
	as a meromorphic function to the whole complex plane
	with, at most, simple poles at the points $s^1_j=d-j$, $j=0,1,2,\dots$

	\item Using now the expansion of $p(-s)$ with respect to $x$, exchanging the role of variable and covariable
	with respect to the previous point, again first assuming $\Re s > d$ and choosing $M\ge1$, one can write
	\begin{align*}
        \zeta_3(s)&= (2\pi)^{-d}
            \sum_{k=0}^{M-1}\frac{1}{s-(d-k)}\int_{\mathbb{S}^{d-1}} \int_{|\xi|\le 1} p_{\cdot, -s-k}(\theta, \xi;-s)d\xi d\theta \\
        &+(2\pi)^{-d}\iint_{\Omega_3} t_{\cdot, -s-M}(x,\xi;-s) d\xi dx.
	\end{align*}
    Arguing as in point \ref{it:zeta2}, $\zeta_3(s)$ turns out to be holomorphic for $\Re s>d$, 
	extendable as a meromorphic function to the whole complex plane with, 
	at most, simple poles at the points $s^2_k= d-k$, $k=0,1,2,\dots$

	\item To treat the last term, both the expansions with respect to $x$ and with respect to $\xi$ are needed.
	    We assume that $\Re s > d$ and choose $M \ge 1$. We argue as in point \ref{it:zeta2} to obtain
	\begin{align*}
	\zeta_4(s)&= (2\pi)^{-d}\sum_{j=0}^{M-1} \frac{1}{s-(d-j)}\int_{|x|\geq 1} \int_{\mathbb{S}^{d-1}}p_{-s-j, \cdot}(x, \omega;-s)d\omega dx\\
    &\phantom{=} + (2\pi)^{-d}\iint_{\Omega_4} r_{-s-M, \cdot}(x, \xi;-s)d\xi dx.
    \end{align*}
    Now, we introduce the expansion with respect to $x$, switching to polar coordinates and integrating the $x$-radial variable
    in the homogeneous terms, for both integrals
    \begin{align*}
        \int_{|x|\geq 1}\int_{\mathbb{S}^{d-1}}p_{-s-j, \cdot}(x, \omega;-s)d\omega dx &= \sum_{k=0}^{M-1} \frac{1}{s-(d-k)} \int_{\Sph^{d-1}}\int_{\mathbb{S}^{d-1}} p_{-s-j, -s-k}(\theta, \omega;-s) d\theta d\omega\\
        &\phantom{=} + \int_{|x|\geq 1} \int_{\mathbb{S}^{d-1}} t_{-s-j, -s-M}(x, \omega;-s) dx d\omega
    \end{align*}
    and
    \begin{align*}
        \iint_{\Omega_4} r_{-s-M, \cdot}(x, \xi;-s)d\xi dx &= 
        \sum_{k=0}^{M-1} \frac{1}{s-(d-k)}\int_{\mathbb{S}^{d-1}}\int_{|\xi|\ge1} r_{-s-M, -s-k}(\theta,\xi; -s) d\xi d\theta\\
        &\phantom{=} + \iint_{\Omega_4}r_{-s-M, -s-M}(x, \xi; -s)dx d\xi.
    \end{align*}
    We end up with
    \begin{align*}
	\zeta_4(s) &= \sum_{k=0}^{M-1} \sum_{j=0}^{M-1} \frac{1}{s-(d-j)} \frac{1}{s-(d-k)} I_{j}^{k}(s)\\
    &\phantom{=} + \sum_{j=0}^{M-1} \frac{1}{s-(d-j)} R_j^M(s) +\sum_{k=0}^{M-1} \frac{1}{s-(d-k)} R_M^k(s) + R^M_M(s),
	\end{align*}
	where 
	\[
		I_{j}^{k}(s)=(2\pi)^{-d} \int_{\mathbb{S}^{d-1}}\int_{\mathbb{S}^{d-1}} p_{-s-j, -s-k}(\theta', \theta; -s)d\theta d\theta',
	\]
	and $R^j_M$, $R^M_k$, $R_M^M$, are holomorphic in $s$ for $\Re s > M + d$, $j,k=0,\dots, M-1$. 
	It follows that $\zeta_4(s)$ is holomorphic for $\Re(s)>d$ and can be extended 
	as a meromorphic function to the whole complex plane with, at most, poles at the points 
	$s^1_j=d-j$, $s^2_k= d-k$ with $j,k \in \NN_0$. Clearly, such poles can be of order two if and only if $j=k$
	(cf. Theorem \ref{thm:fzeta}).
\end{enumerate}
In view of the properties of $\zeta(s)$ recalled above, the limit
\[
	\lim_{s\to d-1}[s-(d-1)]^2\zeta(s)=\lim_{s\to d-1}[s-(d-1)]^2\zeta_4(s)=I^1_1(d-1)
\]
proves the desired claim.
\end{proof}

\section{Proof of the Main Theorems}\label{sec:proofs}
Arguing as in \cite{CoMa13}, it is enough to prove 
Theorems \ref{thm:sharp-weyl} and \ref{thm:refined} for $P\in\calcclsg^{1,1}(\RR^d)$.
In such situation, as explained in \cite{BaCo11},
    \begin{align} \label{eq:tra}
        &\TR(P^{-d})
        =(2\pi)^{-d} \int_{\Sph^{d-1}}\int_{\Sph^{d-1}} p_{\psie}(\theta,\omega)^{-d} d\theta\,d\omega,
    \end{align}
    \begin{align} \label{eq:trb}
       d \cdot \wTR&_{x,\xi}(P^{-d}) - \TR(P^{-d})\\ \nonumber
        =(2\pi)^{-d} &\int_{\Sph^{d-1}}\int_{\Sph^{d-1}} p_{\psie}(\theta,\omega)^{-d} \log \left(p_\psie(\theta,\omega)^{-d}\right)  d\theta \, d\omega \\ \nonumber
        -(2\pi)^{-d} &\lim_{\tau\to+\infty}\left[
            \int_{|x|\le\tau}\int_{\Sph^{d-1}}p_{\psi}(x,\omega)^{-d}dx\,d\omega -(\log\tau)\int_{\Sph^{d-1}}\int_{\Sph^{d-1}} p_\psie(\theta,\omega)^{-d} d\theta \, d\omega
        \right]
            \\ \nonumber
        -(2\pi)^{-d} &\lim_{\tau\to+\infty}\left[
            \int_{\Sph^{d-1}}\int_{|\xi|\le\tau}p_{e}(\theta,\xi)^{-d} d\theta\,d\xi
            -(\log\tau)\int_{\Sph^{d-1}}\int_{\Sph^{d-1}} p_\psie(\theta,\omega)^{-d} d\theta \, d\omega
        \right]
            \\ \nonumber
        - (2\pi)^{-d}&\int_{\Sph^{d-1}}\int_{\Sph^{d-1}} p_\psie(\theta,\omega)^{-d} d\theta \, d\omega,
    \end{align}
    where the triple $(p_\psi,p_e,p_\psie)$ is the principal symbol of $P$.

We choose a positive function $\rho \in \Sw(\RR)$ such that $\hat\rho(0) = 1$, $\supp \hat \rho \subset [-1,1]$, and $\rho$ is even.
For $T > 0$, we set $\rho_T(\lambda) \coloneqq T \rho(T\lambda)$, which implies that $\hat\rho_T(t) = \hat\rho(t/T)$.
Let $\nu > 0$ be arbitrary. Then, it is possible to prove the next Tauberian theorem by following the proof in \cite{SaVa}*{Appendix B}.
\begin{theorem}[Tauberian theorem] \label{thm:tauberian}
    Let $N : \RR \to \RR$ such that $N$ is monotonically nondecreasing, $N(\lambda) = 0$ for $\lambda \leq 0$,
    and is polynomially bounded as $\lambda \to +\infty$.
    If
    \begin{align*}
        (\partial_\lambda N * \rho_T)(\lambda) \leq C_1 \lambda^\nu \log\lambda, \quad\lambda \geq T^{-1}
    \end{align*}
    for $C_1 > 0$,
    then
    \begin{align*}
        \abs{N(\lambda) - (N * \rho_T)(\lambda)} \leq C\,C_1 T^{-1} \lambda^\nu \log\lambda, \quad\lambda \geq T^{-1}.
    \end{align*}
\end{theorem}

\begin{proof}[Proof of Theorem~\ref{thm:sharp-weyl}]
    The first part of Theorem \ref{thm:sharp-weyl} follows directly from the Tauberian theorem and Proposition~\ref{prop:micro-wavetrace}, due to the identity
    \begin{align*}
        [\F(N^\prime)](t) = \Tr e^{-itP}.
    \end{align*}
    From Proposition~\ref{prop:zeta} it follows that the coefficients $w_{j,k}$ are given by the Laurent coefficients of $\zeta(s)$.
\end{proof}
To prove Theorem~\ref{thm:refined} it suffices to prove that
\begin{align*}
    N(\lambda) = (N*\rho)(\lambda) + o(\lambda^{d-1}\log\lambda),
\end{align*}
where $(N*\rho)(\lambda)$ is obtained through Propositions \ref{prop:micro-wavetrace} and \ref{prop:zeta}.
We define the microlocal return time function $\Pi : \Wt \to \RR_+ \cup \{\infty\}$ by
\begin{align*}
    \Pi(x,\xi) = \inf \{t > 0 \colon \exp(t\hamvf_{\sigma(P)})(x,\xi) = (x,\xi)\},
\end{align*}
and $\Pi(x,\xi) = \infty$ if no such $t$ exists. For a set $\Gamma \in \Wt$, we set $\Pi_\Gamma = \inf_{z \in \Gamma} \Pi(z)$.

We will need a microlocalized version of the Poisson relation.
\begin{proposition}\label{prop:poisson-micro}
    Let $\Gamma \subset \Wt$
    and $\hat\chi \in \Smc(\RR)$ with $\supp \hat\chi \subset (0,\Pi_\Gamma)$.
    For all $B \in \Op\SG_\cl^{0,0}$ with $\WFSG'(B) \subset \Gamma$, we have that
    \begin{align*}
        \hat\chi(t) \Tr (U(t) B B^*) \in \Smc(\RR).
    \end{align*}
    In particular, $(\chi * N_B')(\lambda) \in O(\lambda^{-\infty}).$
\end{proposition}
The proof is a standard argument (cf. Wunsch~\cite{Wunsch99}) and is only sketched here.
\begin{proof}[Proof of Proposition \ref{prop:poisson-micro}]
    For $t_0 \in \supp \hat\chi$ and $(x,\xi) \in \Gamma$, we choose a conic neighborhood $U$ of $(x,\xi)$ such that
    \begin{align*}
        [\Phi(t) U] \cap U = \emptyset
    \end{align*}
    for all $t \in (t_0-\ep,t_0+\ep)$ with $\ep > 0$ sufficiently small.
    The existence of this neighborhood is guaranteed by the conditions on $\Gamma$ and $\supp \hat\chi$.
    Choose $\tilde B \in \Op\SG_\cl^{0,0}$ with $\WFSG'(\tilde B) \subset U$.
    Lemma~\ref{lem:prop-sing} implies that for any $k \in \NN$,
    \begin{align*}
        \pa_t^k \left(\tilde{B} U(t) \tilde{B}\right) = \tilde{B} P^k U(t) \tilde{B} \in \mathcal{L}(\Sw'(\RR^d),\Sw(\RR^d)),
    \end{align*}
    hence $\tilde{B} U(t) \tilde{B}$ and all its derivatives are trace-class.
    We obtain the claim by using a partition of unity.
\end{proof}

We also define the modified return time \[\tilde{\Pi}(x,\xi) = \max\{\Pi(x,\xi), \ep\},\]
where $\ep$ is given as in~\eqref{eq:wave-parametrix}, and set $\tilde{\Pi}_\Gamma=\inf_{z\in\Gamma}\tilde{\Pi}(z)$.
The main tool to prove Theorem~\ref{thm:refined} is the next Proposition \ref{prop:error}.
\begin{proposition}\label{prop:error}
    It holds true that
     \begin{align*}
       \limsup_{\lambda \to \infty} \frac{|N(\lambda) - (N * \rho)(\lambda)|}{\lambda^{d-1}\log\lambda}
       \leq C \int_{\Wpe} \tilde{\Pi}(x,\xi)^{-1} \frac{dS}{p_{1,1}(x,\xi)}.
    \end{align*}
\end{proposition}

\begin{proof}[Proof of Theorem~\ref{thm:refined}]
The claim follows immediately by Proposition~\ref{prop:error}, 
since the assumptions imply that $\Pi(x,\xi)^{-1} = 0$ almost everywhere on $\Wpe$.
\end{proof}

\begin{proof}[Proof of Proposition~\ref{prop:error}]
Consider an open covering $\{\Gamma^\bullet_j\}$ of $\Wt$ with $\bullet \in \{\psi,e,\psie\}$ and $j \in \{1,\cdots, n_\bullet\}$ such that
$\Gamma_j^\psi \subset \Wt^\psi$ and $\Gamma_j^e \subset \Wt^e$ do not intersect $\Wpe$,
and $\Gamma_j^\psie\cap\Wt^\psie\not=\emptyset$.

We consider a partition of unity on the level of operators such that
\begin{align*}
    I = \sum_{j=1}^{n_\psi} A^\psi_j (A^\psi_j)^* + \sum_{j=1}^{n_e} A^e_j(A^e_j)^* + \sum_{j=1}^{n_\psie} A^\psie_j(A^\psie_j)^* + R,
\end{align*}
where $A^\psi_j \in \Op\SG_\cl^{0,-\infty}$, $A^e_j \in \Op\SG_\cl^{-\infty,0}$, $A^\psie_j \in \Op\SG_\cl^{0,0}$ and $R \in \mathcal{L}(\Sw',\Sw)$. Furthermore, we assume that $\WF(A^\bullet_j) \subset \Gamma_j^\bullet$.

Inserting the partition of unity into the counting function yields
\begin{align*}
    N(\lambda) = \sum_{j=1}^{n_\psi}N^\psi_j(\lambda) + \sum_{j=1}^{n_e} N^e_j(\lambda) + \sum_{j=1}^{n_\psie} N^\psie_j(\lambda) + \Tr(E_\lambda R),
\end{align*}
where as before $N^\bullet_j(\lambda) = \Tr E_\lambda A^\bullet_j (A^\bullet_j)^* = \sum_{\lambda_k < \lambda} \|(A^\bullet_j)^* \psi_k\|^2$. Here, $\psi_k$ are the eigenfunctions of $P$ with eigenvalue $\lambda_k$.

By the classical result of Hörmander~\cite{Hormander68}, we have that $N^\psi_j(\lambda) = (N^\psi_j * \rho)(\lambda) + O(\lambda^{d-1})$
and by \cite{CoMa13} we obtain that $N^e_j(\lambda) = (N^e_j * \rho)(\lambda) + O(\lambda^{d-1})$.
The operator $E_\lambda R$ is regularising, thus
its trace is uniformly bounded. We arrive at
\begin{align*}
    N(\lambda) - (N * \rho)(\lambda) = \sum_{j=1}^{n_\psie} \left[N^\psie_j(\lambda) - (N^\psie_j*\rho)(\lambda)\right] + O(\lambda^{d-1}).
\end{align*}

It remains to estimate the terms $N^\psie_j(\lambda) - (N^\psie_j*\rho)(\lambda)$.
For this let
\[
    \Pi_j = \inf_{(x,\xi) \in \Gamma_j^\psie}\Pi(x,\xi),\quad
    \tilde{\Pi}_j = \max\{ \Pi_j, \ep\}.
\]

For $1/T < \ep$, we have by Proposition~\ref{prop:micro-wavetrace} that
\begin{align*}
    (N^\psie_j * \rho)(\lambda) \sim \sum_{k=0}^\infty \sum_{j=0,1} w_{jk} \lambda^{d-k} (\log\lambda)^j.
\end{align*}
This implies that the derivative is given by
\begin{align*}
    (\partial_\lambda N^\psie_j * \rho)(\lambda) = d \cdot w_{1,0} \lambda^{d-1}\log\lambda + O(\lambda^{d-1}),
\end{align*}
where $w_{1,0}$ is given by Proposition~\ref{first-term}. Namely,
\begin{align*}
    w_{1,0} = \frac{1}{d} \int_{\Sph^{d-1}}\int_{\Sph^{d-1}} [p_{\psie}(\theta,\omega)]^{-d} \cdot |\sigma^\psie(A^\psie_j)(\theta,\omega)|^2  d\theta\,d\omega.
\end{align*}
Together with Proposition~\ref{prop:poisson-micro} this implies that
\begin{align*}
    (N^\psie_j * \rho_T)(\lambda) &= (N^\psie_j * \rho)(\lambda) + O(\lambda^{-\infty})\\
    &= \sum_{k=0}^\infty \sum_{j=0,1} w_{jk} \lambda^{d-k} (\log\lambda)^j + O(\lambda^{-\infty})
\end{align*}
for $1/T < \tilde\Pi_j$.

Applying the Tauberian theorem to $N^\psie_j * \rho_T$ yields
\begin{align*}
    \frac{|N^\psie_j(\lambda) - (N^\psie_j * \rho)(\lambda)|}{\lambda^{d-1}\log\lambda} \lesssim \tilde\Pi_j^{-1} \int_{\Sph^{d-1}} \int_{\Sph^{d-1}} |\sigma^\psie(A_j^\psie)(\omega,\theta)|^2 p_{1,1}(\omega,\theta)^{-d} d\theta\,d\omega
\end{align*}
for $\lambda \geq \tilde\Pi_j$.
Taking the $\limsup$ and summing over all $j$ gives
\begin{align*}
    \limsup_{\lambda\to \infty}\frac{|N(\lambda) - (N * \rho)(\lambda)|}{\lambda^{d-1}\log\lambda} \lesssim \sum_{j=1}^{n_\psie} \tilde\Pi_j^{-1} \int_{\Sph^{d-1}} \int_{\Sph^{d-1}} |\sigma^\psie(A_j^\psie)(\omega,\theta)|^2 p_{1,1}(\omega,\theta)^{-d} d\theta\,d\omega.
\end{align*}
The right hand side is an upper Riemann sum, therefore we obtain the claim by shrinking the partition of unity.
\end{proof}

\section{An example: the model operator \texorpdfstring{$P=\ang{\cdot}\ang{D}$}{P=<x><D>}}\label{sec:ex}
In this section, we will consider the case of the operator $P = \ang{\cdot}\ang{D}$ on $\RR^d$. 
First we compute the full symbol of $P$ near the corner:
\begin{align*}
	\ang{x}\ang{\xi}&=	|x|\cdot|\xi|\cdot\left(1+\frac{1}{|x|^2}\right)^\frac{1}{2}\left(1+\frac{1}{|\xi|^2}\right)^\frac{1}{2}
	\\
	&=|x|\cdot|\xi|\cdot\sum_{j,k=0}^\infty\begin{pmatrix}\frac{1}{2}\\j\end{pmatrix}\begin{pmatrix}\frac{1}{2}\\k\end{pmatrix}
	(-1)^{j+k}\frac{1}{|x|^{2j} \cdot |\xi|^{2k}}
	\\
	&=\sum_{j,k=0}^\infty\begin{pmatrix}\frac{1}{2}\\j\end{pmatrix}\begin{pmatrix}\frac{1}{2}\\k\end{pmatrix}
	(-1)^{j+k}|x|^{1-2j} \cdot |\xi|^{1-2k}.
\end{align*}
It follows that $p_\psie(x,\xi) =\sigma^\psie(P)(x,\xi)= |x||\xi|$, $p_\psi(x,\xi) = |\xi| \ang{x}$, and $p_e(x,\xi) = |x|\ang{\xi}$.

We have to investigate the flow of the principal symbol $p_\psie$ in the corner.
The Hamiltonian vector field on $\RR^{2d}$ is given by
\begin{align*}
    \hamvf_{p_\psie} = \pa_\xi p_\psie \pa_x - \pa_x p_\psie \pa_\xi.
\end{align*}
First, we show that the angle between $x$ and $\xi$ is invariant under the flow.
This follows from
\begin{align*}
    \pa_t \ang{x,\xi} &= \ang{\pa_t x, \xi} + \ang{x, \pa_t \xi}\\
    &= \frac{|x|}{|\xi|} \ang{\xi,\xi} - \frac{|\xi|}{|x|} \ang{x,x}\\
    &= |x| |\xi| - |x| |\xi| = 0.
\end{align*}
Hence, the quantity
 \[
 c = c(x_0,\xi_0) = \frac{\ang{x_0,\xi_0}}{|x_0| |\xi_0|}
 \] 
 is preserved by the flow. The Hamiltonian flow $\Phi^\psie(t) : \Wt^\psie \to \Wt^\psie$ is given by the angular part.
\begin{lemma}
    The differential equation for $\omega = x/\abs{x}$ and $\theta = \xi/\abs{\xi}$ describing 
    the Hamiltonian flow $\Phi^\psie(t) : \Wt^\psie \to \Wt^\psie$ is given by
    \begin{equation}\label{eq:omega-theta}
        \left\{
        \begin{aligned}
            \pa_t \omega &= -c\omega + \phantom{c}\theta \\
            \pa_t \theta &= \hspace*{5pt}-\omega + c\theta.\\
        \end{aligned}\right.
    \end{equation}
\end{lemma}
\begin{proof}
    We observe that
    \begin{align*}
        \pa_t \frac{x(t)}{\abs{x(t)}} &= \frac{\pa_t x(t)}{\abs{x(t)}} - \frac{x(t) \pa_t \abs{x(t)}}{\abs{x(t)}^2}.
    \end{align*}
    The calculation of $\pa_t |x|$ is straightforward:
    \begin{align*}
        \pa_t \abs{x} &= \frac{\ang{x,\xi}}{\abs{x} \abs{\xi}} \cdot \abs{x} = \frac{\ang{x_0,\xi_0}}{\abs{x_0} \abs{\xi_0}} \cdot \abs{x},
    \end{align*}
    This implies 
    \begin{align*}
        \pa_t \frac{x(t)}{\abs{x(t)}} = \frac{\xi(t)}{\abs{\xi(t)}} - c \, \frac{x(t)}{\abs{x(t)}},
    \end{align*}
    as claimed. The second equation follows likewise.
\end{proof}
\begin{proposition}\label{prop:hamflowmodel}
    The return time function $\Pi : \Wt^\psie \to \RR$ is given by
    \begin{align*}
        \Pi(\omega,\theta) =
        \begin{cases}
            \dfrac{2\pi}{\sqrt{1 - \ang{\omega,\theta}^2}}, & \ang{\omega,\theta}^2 \not = 1\\
            0, & \ang{\omega,\theta}^2 = 1.
        \end{cases}
    \end{align*}
\end{proposition}
\begin{proof}
    The system of differential equations~\eqref{eq:omega-theta} decomposes into $d$ decoupled systems of the form
    \begin{align*}
        \pa_t \upsilon(t) = A \upsilon(t),
    \end{align*}
    where
    \begin{align*}
        A = \begin{pmatrix} -c & 1 \\ -1 & c\end{pmatrix}
    \end{align*}
    We note that the eigenvalues of the matrix $A$
    are given by $\lambda_\pm = \pm i\sqrt{1 - c^2}$.
    Thus, we have that the fundamental solution to the differential equation~\eqref{eq:omega-theta} for $(\omega, \theta)$ is given by
    \begin{align*}
        S \cdot \begin{pmatrix}e^{-it\sqrt{1-c^2}}\id_d & 0\\0 & e^{it\sqrt{1-c^2}}\id_d\end{pmatrix} \cdot S^{-1}
    \end{align*}
    for some unitary matrix $S = S(c)$.
    The claim follows by choosing the minimal $t > 0$ with $t\sqrt{1-c^2} \in 2\pi\ZZ$
    and noting that $c = \ang{\omega(0),\theta(0)}=\ang{\omega_0,\theta_0}$ for $\omega_0, \theta_0 \in \Sph^{d-1}$.
\end{proof}
\begin{remark}
    Proposition \ref{prop:hamflowmodel} shows that Theorem~\ref{thm:refined} cannot be applied to $P$.
\end{remark}

\begin{proof}[Proof of Theorem~\ref{thm:explicit-example}]
    By the Weyl law, Theorem~\ref{thm:sharp-weyl}, we have that
    \begin{align*}
        N(\lambda) = \gamma_2 \lambda^d \log\lambda + \gamma_1 \lambda^d +  
        O(\lambda^{d-1} \log\lambda).
    \end{align*}
    So it remains to calculate the corresponding Laurent coefficients of $\zeta(s)$.
    With the notation and the results of \cite{BaCo11}, in view of \eqref{eq:tra}, we have that
    \begin{align*}
        \gamma_2 &= \frac{\TR(P^{-d})}{d} 
                             = \frac{(2\pi)^{-d}}d \int_{\Sph^{d-1}}\int_{\Sph^{d-1}} p_\psie(\theta,\omega)^{-d} d\theta \, d\omega
                             = \frac{(2\pi)^{-d}}d \int_{\Sph^{d-1}}\int_{\Sph^{d-1}} d\theta \, d\omega
                             \\
                             &= \frac{[\vol(\Sph^{d-1})]^2}{(2\pi)^{d}} \frac{1}{d}.
    \end{align*}
    The computation of $\gamma_1$ requires a few more considerations. Again with the notation and the results of \cite{BaCo11},
    by \eqref{eq:trb},
    \begin{align*}
        \gamma_1 &= \frac{\wTR_{x,\xi}(P^{-d})}{d} - \frac{\TR(P^{-d})}{d^2}
        = \wTR_\theta(P^{-d})
        - \wTR_\psi(P^{-d}) - \wTR_e(P^{-d})
        - \frac{\TR(P^{-d})}{d^2}
    \end{align*}
    First, we note that $\wTR_\psi(P^{-d}) = \wTR_e(P^{-d})$ and the last term we already calculated for $\gamma_2$.
    We recall that $p_\psie = 1$ on $\Sph^{d-1} \times \Sph^{d-1}$.
    Thus, we have for $\wTR_\theta(P^{-d})$ that
    \begin{align*}
        \wTR_\theta(P^{-d}) &= \frac{1}{(2\pi)^d} \int_{\Sph^{d-1}} \int_{\Sph^{d-1}} p_\psie(\theta,\omega)^{-d} \log \left( p_\psie(\theta,\omega)^{-d} \right) d\theta \, d\omega\\
        &= 0.
    \end{align*}
    This implies
    \begin{align}\label{eq:gamma-wtr}
        \gamma_1 = -2 \cdot \wTR_e(P^{-d}) - \frac{\TR(P^{-d})}{d^2}.
    \end{align}
    Hence, we only have to calculate $\wTR_e(P^{-d})$:
    \begin{align*}
        \wTR_e(P^{-d}) &=\frac{1}{(2\pi)^d} \lim_{\tau\to+\infty}\Big\{ \int_{\Sph^{d-1}}\int_{|\xi|\le\tau}p_e(\theta,\xi)^{-d} d\theta\,d\xi \\
        &\phantom{=\frac{1}{(2\pi)^d} \lim_{\tau\to+\infty} \Big\{ } % this line only contains the phantom text
             -(\log\tau)\int_{\Sph^{d-1}}\int_{\Sph^{d-1}} p_\psie(\theta,\omega)^{-d} d\theta \, d\omega\Big\}\\
        &= \frac{\vol(\Sph^{d-1})^2}{(2\pi)^d} \lim_{\tau\to+\infty}\left[ \vol(\Sph^{d-1})^{-1}\int_{|x|\le\tau}\ang{x}^{-d}dx-\log\tau\right].
    \end{align*}
    Using polar coordinates, we see that
    \begin{align*}
        \vol(\Sph^{d-1})^{-1}\int_{|x|\le\tau}\ang{x}^{-d}dx = \int_0^\tau (1 + r^2)^{-d/2} r^{d-1} dr.
    \end{align*}
    Now, we perform a change of variables $r=t^{-\frac{1}{2}}\Leftrightarrow t=r^{-2}>0$, so that
    \begin{align*}
        \int_0^\tau (1+r^2)^{-\frac{d}{2}}r^{d-1}\,dr &= \frac12 \int_{\tau^{-2}}^{+\infty} (t+1)^{-\frac{d}{2}}t^{-1}\,dt\\
        &= \frac12 \int_{\tau^{-2}}^{+\infty}\frac{dt}{t(t+1)} - \frac12 \int_{\tau^{-2}}^{+\infty} \left[(1+t)^{-1}-(t+1)^{-\frac{d}{2}}\right] \frac{dt}{t}
    \end{align*}
    For $\Re z > 0$, we have that (cf. \cite{GrRy80}*{\#8.36})
    \[
    	\Psi(z)=\int_0^{+\infty} \left[(1+t)^{-1}-(t+1)^{-z}\right] \frac{dt}{t}-\gamma,
    \]
    and, by elementary computations,
    \begin{align*}
    	\int_{\tau^{-2}}^{+\infty}\frac{dt}{t(t+1)}+\log\tau^{-2}&=
	\lim_{\kappa\to+\infty}\left[
		\log\frac{\kappa}{\kappa+1}-\log\tau^{-2}+\log(1+\tau^{-2})
	\right]+\log\tau^{-2}
	\\
	&=\log(1+\tau^{-2})\longrightarrow 0 \text{ for }\tau\to+\infty.
    \end{align*}
    Hence, we have that
    \begin{align*}
        \lim_{\tau\to+\infty}&\left[ \vol(\Sph^{d-1})^{-1}\int_{|x|\le\tau}\ang{x}^{-d}dx-\log\tau \right]\\
        &= \frac12 \lim_{\tau \to \infty}\left[\int_{\tau^{-2}}^{+\infty}\frac{dt}{t(t+1)} + \log\tau^{-2}\right]- \frac12 \lim_{\tau \to \infty}\int_{\tau^{-2}}^{+\infty} \left[(1+t)^{-1}-(t+1)^{-\frac{d}{2}}\right] \frac{dt}{t}\\
        &= -\frac12 \left[ \Psi(d/2) + \gamma\right].
    \end{align*}
    Summing up, we have obtained
    \begin{equation}\label{eq:gonefinal}
    	\gamma_1 = \frac{[\vol(\Sph^{d-1})]^2}{(2\pi)^{d}} \cdot \!\left[\Psi\!\left(\frac{d}{2}\right)+\gamma-\frac{1}{d^2}\right].
    \end{equation}
    The proof is complete.
\end{proof}

\begin{remark}
	Using the properties of the function $\Psi$, we can make \eqref{eq:gonefinal}
	more explicit. Indeed, see, e.g., \cite{GrRy80}*{\#8.366, page 945}, we find:
	\[
		\gamma_1=
		\begin{cases}
			\displaystyle -
			\frac{[\vol(\Sph^{d-1})]^2}{(2\pi)^{d}} \left(2\log2+\frac{1}{d^2}-2\sum_{k=1}^{\frac{d-1}{2}}\frac{1}{2k-1}\right)\!, 
			& \text{ if $d$ is odd,}
			\rule{0mm}{11mm}
			\\
			\displaystyle -\frac{[\vol(\Sph^{d-1})]^2}{(2\pi)^{d}} \left(\frac{1}{d^2}-\sum_{k=1}^{\frac{d}{2}-1}\frac{1}{k}\right),
			& \text{ if $d$ is even.}
			\rule{0mm}{11mm}
		\end{cases}
	\]
\end{remark}
In particular, we have that
\begin{align*}
    \gamma_1 = \begin{cases} -\dfrac2\pi (2 \log 2 + 1), & d=1,\\
        -\dfrac14, & d=2.\rule{0mm}{8mm}
    \end{cases}
\end{align*}

\appendix
\section{SG-classical operators on asymptotically Euclidean manifolds}\label{subs:scmf}
We refer to \cites{CDS19,Melrosemwc} for a detailed study of scattering geometry.

\begin{definition}
    An asymptotically Euclidean manifold $(X,g)$ is a compact manifold with boundary $X$,
    whose interior is equipped with a Riemannian metric $g$ that is supposed to take the form, in a tubular neighborhood of the boundary,
    \[g=\frac{d\rho^2}{\rho^4}+\frac{g_\partial}{\rho^2},\]
    where $\rho$ is a boundary defining function and $g_\partial\in\Sm(X,\Sym^2 T^*X)$ restricts to a metric on $\partial X$.
\end{definition}

Under the sterographic projection $\SP : x \mapsto \ang{x}^{-1}(1, x) \in \Sph^d$ we may identify
$\RR^d$ with the interior of $\Sph^d_+ = \{ y = (y_0,\dotsc,y_{d+1}) \colon y_0 \geq 0, |y| = 1\}$. If we set $\rho = |x|^{-1}$, then the Euclidean metric becomes
\begin{align*}
    g \cong \frac{d\rho^2}{\rho^4} + \frac{g_{\Sph^{d-1}}}{\rho^2},
\end{align*}
where $g_{\Sph^{d-1}}$ is the induced metric on the sphere.

For any compact manifold with boundary $X$ with boundary defining function $\rho_X$,
we define the space of scattering vector fields $\scV(x) \coloneqq \rho\,\bV(X)$, where $\bV(X)$ is the space tangential vector fields.
There is natural vector bundle, $\scT X$ such that the sections of $\scT X$ are exactly the scattering vector fields. The dual bundle is the
\emph{scattering cotangent bundle}, $\scT^* X$. Using the fiberwise stereographic projection, we obtain a manifold with corners
$\scOverT^* X$ with boundary defining functions $\rho_X$ and $\rho_\Xi$.

The new-formed fiber boundary may be identified with a rescaling of the cosphere bundle, called $\scS^*X$.
Since $X$ is a compact manifold with boundary, $\scOverT^*X$ is a compact manifold with corners.
The boundary $\Wt$ of $\scOverT^*X$ splits into three components:
$$\We:=\scT^*_{\partial X}X,\qquad \Wp:=\scS^*_{X^o}X,\qquad \Wpe:=\scS^*_{\partial X}X.$$
It can be shown (cf. \cite{CoSc13}) that the $\SG$-classical symbols $\SG^{m_\psi,m_e}_\cl(\RR^d)$ become under this identification $\rho_X^{-m_e}\rho_{\Xi}^{-m_\psi}\Sm(\Sph^d_+\times \Sph^d_+)$.
All the concepts and notions introduced in the previous parts of this section, for the \emph{local model} given by $\RR^d$ and its compactification $\Sph^d_+$, extend to the setting of a general scattering manifold $X$.
%

% Hamiltonian vector fields
Melrose--Zworski~\cite{MeZw96} defined for $f \in \rho_X^{-m_e}\rho_\Xi^{-m_\psi} \Sm(\scOverT^*X)$ the Hamiltonian vector field
\begin{align*}
    \scHamvf_f \in \rho_X^{-m_e+1}\rho_\Xi^{-m_\psi+1} \;\bV(\scOverT^*X),
\end{align*}
which generalizes the usual Hamiltonian vector field to the compactified cotangential bundle of asymptotically Euclidean manifolds.

For $f \in \rho_X^{-1}\rho_\Xi^{-1} \Sm(\scOverT^*X)$, the Hamiltonian vector field is tangential to the boundary and hence its flow $\exp(t\,\scHamvf_f)$ can be restricted to a map
\begin{align*}
    \exp(t\,\scHamvf_f)|_\Wt : \Wt \to \Wt
\end{align*}
that preserves the components $\We$, $\Wp$, and $\Wpe$.
Note that the flow $t\mapsto \exp(t\,\scHamvf_f)|_\Wt$ depends only on the principal symbol of $f$.

The propagation of singularities results from \cite{CoMa03} now reads as follows:
\begin{proposition}
    Let $P$ be an elliptic SG-pseudodifferential operator of order $(1,1)$ on an asymptotically Euclidean manifolds $(X,g)$.
    Denote by $\Phi(t) : \Wt \to \Wt$ the Hamiltonian flow associated with the principal symbol of $P$.
    Then
    \begin{align*}
        \WFSG(e^{-itP}u) = \Phi(t) (\WFSG(u)).
    \end{align*}
\end{proposition}

\begin{remark}
	Actually, the results on complex powers, trace operators and spectral asymptotics of $\SG$-classical operators
	have been proved in detail, in \cite{BaCo11} and \cite{CoMa13}, for operators defined on the subclass of manifolds 
	with (cylindrical) ends. In particular, the results about the Cauchy 
	problems for $\SG$-hyperbolic operators of order $(1,1)$ yield there a global parametrix $\widetilde{U}(t)$, 
	locally represented by operators 
	with kernel given in \eqref{eq:utker}, see \cite{CoMa13}. To keep this exposition within a reasonable length, and avoid to
	deviate from our main focus, the detailed analysis of the extension of such previous results to general scattering
	manifolds, as well as the proof of some results on the operator $\SG$-wavefront set, tacitly used above,
	will be illustrated elsewhere.
\end{remark}

\end{document}